\documentclass[11pt,a4paper,leqno]{amsart}
\usepackage{a4wide}
\usepackage{latexsym}
\usepackage{amsmath}
\usepackage{amssymb}
\usepackage{stackrel} 
\usepackage{color}
\usepackage{graphicx}
\usepackage{hyperref}

\newtheorem{lemma}{Lemma}
\newtheorem{prop}{Proposition}
\newtheorem{theo}{Theorem}

\theoremstyle{definition}
\newtheorem{defin}{Definition}
\newtheorem{remark}{Remark}
\newtheorem{example}{Example}
\newcommand{\C}{\mathbb{C}}
\newcommand{\N}{\mathbb{N}}
\newcommand{\R}{\mathbb{R}}
\newcommand{\Q}{\mathbb{Q}}

\newcommand{\balpha}{\boldsymbol{\alpha}}
\newcommand{\bbeta}{\boldsymbol{\beta}}
\newcommand{\bgamma}{\boldsymbol{\gamma}}
\newcommand{\bm}{\boldsymbol{m}}
\newcommand{\bnu}{\boldsymbol{\nu}}

\newcommand{\brho}{\boldsymbol{\rho}}

\newcommand{\bs}{\boldsymbol{s}}

\newcommand{\bz}{\boldsymbol{z}}
\newcommand{\bzeta}{\boldsymbol{\zeta}}

\makeatletter
\def\author@andify{%
  \nxandlist {\unskip ,\penalty-1 \space\ignorespaces}%
    {\unskip {} \@@and~}%
    {\unskip \penalty-2 \space \@@and~}%
}
\makeatother

\begin{document}
\title[Estimates of formal solutions for some generalized moment PDEs]{Estimates of formal solutions for some generalized moment partial differential equations}
\author{Alberto Lastra}
\address{Departamento de F\'isica y Matem\'aticas\\
University of Alcal\'a\\
Ap. de Correos 20, E-28871 Alcal\'a de Henares (Madrid), Spain}
\email{alberto.lastra@uah.es}
\author{S{\l}awomir Michalik}
\address{Faculty of Mathematics and Natural Sciences,
College of Science\\
Cardinal Stefan Wyszy\'nski University\\
W\'oycickiego 1/3,
01-938 Warszawa, Poland}
\email{s.michalik@uksw.edu.pl}
\urladdr{\url{http://www.impan.pl/~slawek}}
\author{Maria Suwi\'nska}
\address{Faculty of Mathematics and Natural Sciences,
College of Science\\
Cardinal Stefan Wyszy\'nski University\\
W\'oycickiego 1/3,
01-938 Warszawa, Poland}
\email{m.suwinska@op.pl}
\date{}
\begin{abstract}
Using increasing sequences of real numbers, we generalize the idea of formal moment differentiation first introduced by W.~Balser and M.~Yoshino. Slight departure from the concept of Gevrey sequences enables us to include a wide variety of operators in our study. Basing our approach on tools such as the Newton polygon and divergent formal norms, we obtain estimates for formal solutions of certain families of generalized linear moment partial differential equations with constant and time variable coefficients.
\end{abstract}

\maketitle
\thispagestyle{empty}

\section{Introduction}

The present work is devoted to the study of a family of Cauchy problems of the form
\begin{equation}
\left\{ \begin{aligned}
P(\partial_{m_0,t}, \partial_{\bm, \bz}) u&=f(t,\bz)\\
\partial_{m_0, t}^j 
u(0,\bz)&=\varphi_j(\bz), \quad j=0,\ldots,K-1.
\end{aligned}
   \right.
	\label{epralintro} 
\end{equation} 
Here $K\ge 1$, and $P$ stands for a polynomial in $N+1$ variables with formal coefficients which belong to a class of formal power series in time variable $t$. The forcing term $f$ turns out to be a formal power series in $t$ with coefficients being functions holomorphic in a common neighborhood of the origin in $\C^N$. The initial conditions of the Cauchy problem are holomorphic functions in a neighborhood of the origin in $\C^N$. 

The operators $\partial_{m_0,t}$ and $\partial_{m_j,z_j}$, with $\partial_{\bm,\bz}=\partial_{m_1,z_1}\cdots \partial_{m_N,z_N}$, fall into a family of formal moment differentiation, which has been put forward by W. Balser and M. Yoshino in~\cite{BY}. Given a sequence $m$ of positive real numbers, the precise definition of an $m$-differential operator $\partial_{m,z}$ is provided in Definition~\ref{defi164}.

Differential operators of this kind generalize not only the usual derivatives in the particular case that $m=(\Gamma(1+n))_{n\ge0}$, but also other differential operators. More precisely, the operator $\partial_{m,z}$ is quite related to the Caputo $1/p-$fractional differential operator $\partial_z^{1/p}$ in the case when $m=(\Gamma(1+\frac{n}{p}))_{n\ge0}$ (we refer the reader to~\cite{M}, Remark 3 for further details). 
Similarly, for $q\in(0,1)$ and for the sequence $m=([n]_q!)_{n\ge0}$, the $m$-differential operator $\partial_{m,t}$ coincides with the $q$-difference operator $D_{q,t}$ (see Example \ref{ex:2}).
As a matter of fact, sequences as above usually come from the moments of some operators, and therefore they are known as moment sequences, whereas the related differential equations are said to be moment differential equations. In recent years the interest in moment differential equations has increased and where several important advances have been accomplished in the framework of asymptotic theory of solutions of such equations. After the seminal work~\cite{BY}, different works have dealt with problems in this direction. It is worth mentioning the Cauchy problems involving two moment sequences studied in~\cite{M}. Such moment sequences corresponded to different kernels (see Section 6.5 in~\cite{balser} for the fundamentals on this theory). 

The works~\cite{michalik13jmaa,michalik17fe} provide further steps in the study of the convergence and summability results for families of homogeneous and inhomogeneous linear moment partial differential equations in two complex variables with constant coefficients.

In~\cite{lastramaleksanz} the authors put forward a definition of summability of formal power series related to the so-called strongly regular sequences and related functional spaces in order to apply their results to the study of the summability properties of the formal solutions to some moment partial differential equations. More recently S. Michalik and B. Tkacz~\cite{michaliktkacz} have studied Stokes phenomenon concerning the solutions of certain families of linear moment partial differential equations with constant coefficients under certain conditions on the Cauchy data. 

This research is focused on the study of Cauchy problems in the form (\ref{epralintro}), where $m_0$ and each of the components in $\bm=(m_1,\ldots,m_N)$ are sequences of positive real numbers under certain assumptions (see Section~\ref{secgenmom}) to be specified. As mentioned above, such properties of the sequences naturally appear while handling spaces of functions and/or formal power series which are subject to such bounds. We refer the reader to~\cite{javi17} for further details on the properties satisfying the moment sequences endowed by the corresponding functional spaces, in the framework of ultraholomorphic classes of functions.

It is important to mention the significance of the development of the formal norms in Section~\ref{secformalnorms}, which allow us to give upper estimates on every formal moment derivative simultaneously.

It is also worth emphasizing that 
generalized moment partial differential equations considered here contain, beside differential or fractional differential equations, also for example
$q$-difference-differential and fractional $q$-difference-differential equations.

The first main result in this work, Theorem~\ref{theo1}, states that if the coefficients in $P$ and forcing term $f(t,\bz)$ in (\ref{epralintro}) are formal power series of some appropriate $(M_n)$-order (see Definition~\ref{def1}), then the formal solution of (\ref{epralintro}) remains in the same space of formal power series, under the hypothesis that $m_0$ is a regular $(M_n)$-sequence and the components of $\bm$  are $(M_n)$-sequences of appropriate order. The relationship among orders and other parameters involved in the problem is provided by the Newton polygon associated with the equation. This result rests strongly on the construction of a family of formal norms, generalizing those found in~\cite{TY,michaliksuwinska}.

In the second main result (Theorem~\ref{theo2}) a similar behavior of the formal solution is observed when dealing with an equation with constant coefficients, under weaker constraints on the sequence $(M_n)_{n\ge0}$ involved in the elements determining the problem, and also under weaker assumption that $m_0$ is a $(M_n)$-sequence, but not necessarily regular. Finally, the multidimensional case is discussed, giving rise to a parallel result, Theorem~\ref{prop4}, under the assumption that the orders of the $(M_n)$-sequences in $\bm$ are rational numbers. It is worth mentioning that this last constraint on the orders is due to technical reasons (see Lemma~\ref{le:11} and Lemma~\ref{le:12}).

The work is structured as follows: After establishing the notation in Section~\ref{secnot}, we define the sequences of numbers and formal power series involved in the construction and growth of the coefficients of the formal solution of the main problem. Some properties of the moment differential operators defining the problem are also described in Section~\ref{secgenmom}. The auxiliary results on the $\bm$-moment formal norms are provided in Section~\ref{secformalnorms}. The precise statement of the main problem is given in Section~\ref{secmom}, where accurate bounds on the formal norms lead to the first main result (Theorem~\ref{theo1}). The particular case of moment equations with constant coefficients in two dimensions is considered in Section~\ref{sectwo}, leading to the result under milder conditions. Section~\ref{secsev} gives answer to the multidimensional framework of the problem. The paper concludes with Section~\ref{secauxlemas}, where some auxiliary lemmas on sequences of numbers are stated. We have left this part at the end of the work for the sake of clarity, and a more fluent reading of the text.

\section{Notation}\label{secnot}

Let us introduce the following notation. Throughout this work $\N$ stands for the set of positive integers, and $\N_0=\N\cup\{0\}$. Let $N\ge 1$. We write $D_R^N$ for the open polydisc in $\C^N$ with a center at $0$ and radius $R>0$, i.e., 
$$
D^N_R=\{(z_1,\dots,z_N)\in\C^N\colon |z_j|<R\ \textrm{for}\ j=1,\dots,N\}.
$$

For every multi-index $\balpha=(\alpha_1,\ldots,\alpha_N)$ and $\bbeta=(\beta_1,\ldots,\beta_N)$ in $\N_0^N$, all $\bz=(z_1,\ldots,z_N)\in\C^N$ and any $A\in\R$, we write
\begin{align*}
\balpha+\bbeta=(\alpha_1+\beta_1,\ldots,\alpha_N+\beta_N)\qquad & |\balpha|=\alpha_1+\ldots+\alpha_N\\
\bz^{\balpha}=z_1^{\alpha_1}\cdots z_N^{\alpha_N}\qquad & \balpha\cdot\bbeta=\alpha_1\beta_1+\ldots+\alpha_N\beta_N\\
A\bz=(Az_1,\ldots,Az_N)\qquad & \boldsymbol{0}=(0,\ldots,0)
\end{align*}

Let $\bm=(m_1,\ldots,m_N)$ where $m_1=(m_1(n))_{n\ge0},\ldots,m_N=(m_N(n))_{n\ge 0}$ are sequences of positive real numbers. For any $\balpha=(\alpha_1,\ldots,\alpha_N)\in\N_0^N$ we write $\partial_{\bm,\bz}^{\balpha}=\partial_{m_1,z_1}^{\alpha_1}\cdots\partial_{m_N,z_N}^{\alpha_N}$, where the operator $\partial_{m_j,z_j}$ stands for the $m_j$-differential operator whose precise definition can be found in Definition~\ref{defi164}. Let $M=(M_n)_{n\ge0}$ be a sequence of positive real numbers. We write 
$$M_{\balpha}^{\bs}=M_{\alpha_1}^{s_1}\cdots M_{\alpha_N}^{s_N},$$
for all $\balpha=(\alpha_1,\ldots,\alpha_N)\in\N_0^N$ and $\bs=(s_1,\ldots,s_N)\in\R^N$.

Let $\bz=(z_1,\ldots,z_N)$. For formal power series with complex coefficients $f(\bz)=\sum_{\balpha\in\N_0^N} f_{\balpha}\bz^{\balpha}\in\C[[\bz]]=\C[[z_1,\dots,z_N]]$ and $g(\bz)=\sum_{\balpha\in\N_0^N} g_{\balpha}\bz^{\balpha}\in\R[[\bz]]$, where $g_{\balpha}\ge0$ for all $\balpha\in\N_0^N$, we write $f(\bz)\ll g(\bz)$ whenever $|f_{\balpha}|\le g_{\balpha}$ for every $\balpha\in\N_0^N$.

\section{Generalized moment differentiation}\label{secgenmom}
In this section we introduce the concept of $m$-moment differentiation, which extends the idea of W. Balser and M. Yoshino \cite{BY}. To this end, we first introduce certain families of sequences of generalized moments $(m(n))_{n\geq 0}$.
\begin{defin}
Let $(M_n)_{n\ge0}$ be a sequence of positive real numbers with $M_0=1$ and $s\in\R$. We say that a sequence of positive real numbers $(m(n))_{n\ge0}$ is an \emph{$(M_n)$-sequence of order $s$} if there exist $\mathfrak{a},\mathfrak{A}>0$ such that
\begin{equation}\label{e1}
\mathfrak{a}^n(M_n)^s\le m(n)\le \mathfrak{A}^n (M_n)^s,\quad n\ge0.
\end{equation}
\end{defin}

\begin{defin}
Let $(M_n)_{n\ge0}$ be a sequence of positive real numbers with $M_0=1$ and $s\in\R$. We say that a sequence of positive real numbers $(m(n))_{n\ge0}$ is a \emph{regular $(M_n)$-sequence of order $s$} if there exist $\mathfrak{a},\mathfrak{A}>0$ such that
\begin{equation}\label{e2}
\mathfrak{a}\left(\frac{M_n}{M_{n-1}}\right)^s\le \frac{m(n)}{m(n-1)}\le \mathfrak{A} \left(\frac{M_n}{M_{n-1}}\right)^s,\quad n\ge1.
\end{equation}
\end{defin}

\begin{remark}
Any regular $(M_n)$-sequence of order $s$ is a $(M_n)$-sequence of order $s$ for the same constants $\mathfrak{a}$ and $\mathfrak{A}$.
\end{remark}

\begin{defin}
Let $(M_n)_{n\ge0}$ be a sequence of positive real numbers with $M_0=1$. We say
\begin{itemize}
\item the sequence $(M_n)_{n\ge0}$ is \emph{of moderate growth (mg)} if there exists a positive constant $B$ such that $M_{n+k}\le B^{n+k}M_nM_k$ for all $n,k\ge0$.
\item the sequence $(M_n)_{n\ge0}$ satisfies the property $(\star)$ if there exists $b>0$ such that $M_{n+k}\ge b^{n+k}M_nM_k$, for all $n,k\ge0$.
\item the sequence is \emph{logarithmically convex (lc)} if $(M_n)^2\le M_{n-1}M_{n+1}$ for all $n\ge1$.
\end{itemize}
\end{defin}

\begin{remark}
Observe that any increasing sequence of positive real numbers $(M_n)_{n\ge0}$ which satisfies (lc) and $M_0=1$ is such that $M_{n+k}\geq M_n M_k$ for all $n,k\in\N_0$ and, in consequence, $(M_n)_{n\ge0}$ satisfies the property $(\star)$.
\end{remark}

\begin{remark}\label{ejstar}
The sequence
$$M_n=\left\{ \begin{aligned}
             n! &\quad \textrm{ if }n \textrm{ is even}\\
             (n-1)! &\quad\textrm{ if } n \textrm{ is odd}\\
             \end{aligned}
   \right.
$$
satisfies ($\star$) and (mg) conditions, but it is not (lc).

Observe also that $(M_n)$ is an $(n!)$-sequence of order $1$, but is not a regular $(n!)$-sequence.
\end{remark}

\begin{defin}\label{def1}
Let $(M_n)_{n\ge0}$ be a sequence of positive real numbers with $M_0=1$ and $s\in\R$. 

Given $\hat{g}(t)=\sum_{n\ge0}a_{n}t^n\in\C[[t]]$, we say that $\hat{g}(t)$ is of \emph{$(M_n)$-order $s$} if there exist $A,B>0$ such that 
$$|a_n|\le A B^n(M_n)^{s},$$
for every $n\ge 0$. We denote the set of all formal power series of $(M_n)$-order $s$ by $\C[[t]]^{(M_n)}_{s}$. 

Analogously, we say that the formal power series $\hat{g}(t,\bz)=\sum_{n\ge0}g_n(\bz)t^n\in\mathcal{O}(D_R^N)[[t]]$ is \emph{of $(M_n)$-order $s$} if for every $r\in(0,R)$ there exist $A,B>0$ such that
$$|g_n(\bz)|\le A B^n (M_n)^s,$$
for every $\bz\in D_r^N$ and all $n\ge0$. We denote the space of all such power series by $\mathcal{O}(D_R^N)[[t]]^{(M_n)}_{s}$.
\end{defin}

\begin{remark}
The previous definition can be naturally extended when dealing with formal power series with coefficients in some complex Banach space $(\mathbb{E},\left\|\cdot\right\|)$.
\end{remark}

In the following definition, $\mathbb{E}$ stands for a complex Banach space.

\begin{defin}\label{defi164}
Let $m=(m(n))_{n\ge0}$ be a sequence of positive real numbers. We define the (formal) \emph{$m$-differential operator} $\partial_{m,z}\colon\mathbb{E}[[z]]\to\mathbb{E}[[z]]$ by
$$\partial_{m,z}\left(\sum_{n\ge0}\frac{u_n}{m(n)}z^n\right)=\sum_{n\ge0}\frac{u_{n+1}}{m(n)}z^n,$$
for every $\sum_{n\ge0}\frac{u_n}{m(n)}z^n\in\mathbb{E}[[z]]$.
\end{defin}

\begin{example}
For every $s>0$ and $p\in\R$, the sequence $m_{s,p}=(\Gamma(1+sn)\prod_{j=0}^{n}\log^{p}(e+j))_{n\ge0}$ satisfies (mg) and (lc) (and therefore, ($\star$)). The case $p=0$ is the classical example of Gevrey sequence of order $s$. More precisely, Gevrey sequence of order $s=1$ corresponds to the usual differentiation $\partial_z$. Moreover, the sequence $m_{s,0}=(\Gamma(1+sn))_{n\ge0}$ for $s>0$ satisfies
$$(\partial_{m_{s,0},z}u)(z^s)=\partial_z^s(u(z^s)),$$
with $\partial_z^s$ being the Caputo fractional derivative of order $s$ (see Remark 3~\cite{M}, and the references therein). The sequence with $s=1$ and $p=-1$ is quite related to the so-called 1+ level, appearing in the study of difference equations (see~\cite{immink}).
\end{example}

\begin{remark}
Observe that for every $\tilde{s}>0$ it holds that the sequence $m_{\tilde{s}s,\tilde{s}p}$ is a regular $m_{s,p}$- sequence of order $\tilde{s}$.
\end{remark}

\begin{example}
 \label{ex:2}
 Suppose that $q\in(0,1)$ and $m=([n]_q!)_{n\ge0}$, where 
 $[n]_q!=[1]_q[2]_q\cdots[n]_q$ and $[n]_q=1+q+\dots+q^{n-1}$.
 Then the operator $\partial_{m,t}$ coincides with the $q$-derivative
 $D_{q,t}$ defined as $D_{q,t} f(t)=\frac{f(qt)-f(t)}{qt-t}$.
 Moreover, since $1\le [n]_q\le \frac{1}{1-q}$ for any $n\in\N$,
 $m$ is a regular $(M_n)$-sequence of order $0$ for any sequence $(M_n)_{n\ge0}$ of positive numbers.
 Hence the family of equations (\ref{epralintro}) contains in particular linear $q$-difference-differential equations with time variable coefficients in the form $P(D_{q,t}, \partial_{\bz})u=f(t,\bz)$ (see, for example, \cite{ichinobeadachi}).
\end{example}

We may estimate $m$-moment derivatives as follows
\begin{prop}\label{prop1}
Let $(M_n)_{n\ge0}$ be a sequence of positive real numbers, with $M_0=1$, satisfying (mg) condition. Suppose that   $\bm=(m_1(n),\dots,m_N(n))_{n\ge 0}$, where $(m_1(n))_{n\ge 0}$, ..., $(m_N(n))_{n\ge 0}$ are $(M_n)$-sequences of non-negative orders $s_1,\dots,s_N$, respectively. We write $\bs=(s_1,\dots,s_N)$. 

Let $\varphi(\bz)\in\C[[\bz]]$ be such that
$$\varphi(\bz)\ll C\sum_{\balpha\in\N_0^N}(D\bz)^{\balpha},$$
for some $C,D>0$. Then, there exist $D',H>0$ such that
\begin{equation}\label{e4}
\partial_{\bm,\bz}^{\bbeta}\varphi(\bz)\ll CH^{|\bbeta|}M_{\bbeta}^{\bs}\sum_{\balpha\in\N_0^N}(D'\bz)^{\balpha},
\end{equation}
for every multi-index $\bbeta\in\N_0^N$.
\end{prop}

\begin{proof}
Let $\bbeta=(\beta_1,\ldots,\beta_N)\in\N_0^N$. Direct computations show that
$$\partial_{\bm,\bz}^{\bbeta}\varphi(\bz)\ll\partial_{\bm,\bz}^{\bbeta}\sum_{\balpha\in\N_0^N}CD^{|\balpha|}\bz^{\balpha}=C\sum_{\balpha\in\N_0^N}D^{|\balpha+\bbeta|}\frac{m_1(\alpha_1+\beta_1)\dots m_N(\alpha_N+\beta_N)}{m_1(\alpha_1)\dots m_N(\alpha_N)}\bz^{\balpha},$$
with $\balpha=(\alpha_1,\ldots,\alpha_N)$.
Regarding (\ref{e1}) and (mg) property of $(M_n)_{n\ge0}$, and putting $s^*=\max\{s_1,\dots,s_N\}$ we have
\begin{multline*}
\sum_{\balpha\in\N_0^N}CD^{|\balpha+\bbeta|}\frac{m_1(\alpha_1+\beta_1)\dots m_N(\alpha_N+\beta_N)}{m_1(\alpha_1)\dots m_N(\alpha_N)}\bz^{\balpha}\ll\sum_{\balpha\in\N_0^N}CD^{|\balpha+\bbeta|}\frac{\mathfrak{A}^{|\balpha+\bbeta|}}{\mathfrak{a}^{|\balpha|}}\frac{M_{\balpha+\bbeta}^{\bs}}{M_{\balpha}^{\bs}}\bz^{\balpha}\\
\ll\sum_{\balpha\in\N_0^N}C(D\mathfrak{A}B^{s^*})^{|\balpha+\bbeta|}\mathfrak{a}^{-|\balpha|}M_{\bbeta}^{\bs}\bz^{\balpha}, 
\end{multline*} 
which yields (\ref{e4}) for $H=D\mathfrak{A}B^{s^*}$ and $D'=D\mathfrak{A}B^{s^*}\mathfrak{a}^{-1}$.
\end{proof}

\section{Formal norms}\label{secformalnorms}
The crucial role in our study is played by the $\bm$-moment formal norm, which allows us to keep estimations of all moment derivatives together. This tool is an analogue of the divergent formal norm used in Gevrey estimations (see \cite{TY} and \cite{michaliksuwinska}).
\begin{defin}
Let $(M_n)_{n\ge0}$ be a sequence of positive real numbers with $M_0=1$, and
let $\bm=(m_1(n),\dots,m_N(n))_{n\ge 0}$, where $(m_1(n))_{n\ge 0}$, ..., $(m_N(n))_{n\ge 0}$ are $(M_n)$-sequences of non-negative orders $s_1$, ..., $s_N$, respectively. We define the \emph{$\bm$-moment formal norm} of $f(\bz)\in\mathcal{O}(D_R^N)$ by
$$\left\|f(\bz)\right\|_{\brho}=\sum_{\balpha\in\N_0^N}\frac{|\partial_{\bm,\bz}^{\balpha}f(\bz)|}{M_{\balpha}^{\bs}}\brho^{\balpha},$$
for $\bz\in D_R^N$, $\brho\in\C^N$ and $\bs=(s_1,\dots,s_N)$.
\end{defin}
In order to estimate the $\bm$-moment formal norms we need to introduce the formal power series 
$$\Theta^{(\bbeta)}(\brho):=\sum_{\balpha\in\N_0^N}\frac{M_{\balpha+\bbeta}^{\bs}}{M_{\balpha}^{\bs}}{\brho}^{\balpha}$$
for any $\bbeta\in\N_0^N$ and $\brho\in\C^N$.

The following results involve properties, which can be derived from the definition of the previous formal power series.

\begin{lemma}\label{lema41prima}
Let $(M_n)_{n\ge0}$ be a sequence of positive real numbers with $M_0=1$, which satisfies the (lc) condition. Let $\bbeta\in\N_0^N$ and $\bs\in\overline{\R}_+^N$. Then for all $\bgamma\in\N_0^N$ we have 
$$\Theta^{(\bbeta)}(\brho)\ll\frac{M_{\bbeta}^{\bs}}{M_{\bgamma+\bbeta}^{\bs}}\Theta^{(\bgamma+\bbeta)}(\brho).$$
\end{lemma}
\begin{proof}
The proof is based on the (lc) property of $(M_n)_{n\ge0}$. More precisely, we have
\begin{align*}
\Theta^{(\bbeta)}(\brho)&=\sum_{\balpha\in\N_0^N}\frac{M_{\balpha+\bbeta}^{\bs}}{M_{\balpha}^{\bs}}{\brho}^{\balpha}=\sum_{\balpha\in\N_0^N}\frac{M_{\balpha+\bbeta}^{\bs}}{M_{\balpha+\bbeta+\bgamma}^{\bs}}\frac{M_{\balpha+\bbeta+\bgamma}^{\bs}}{M_{\balpha}^{\bs}}{\brho}^{\balpha}\\
&\ll \frac{M_{\bbeta}^{\bs}}{M_{\bbeta+\bgamma}^{\bs}}\sum_{\balpha\in\N_0^N}\frac{M_{\balpha+\bbeta+\bgamma}^{\bs}}{M_{\balpha}^{\bs}}{\brho}^{\balpha}=\frac{M_{\bbeta}^{\bs}}{M_{\bbeta+\bgamma}^{\bs}}\Theta^{(\bbeta+\bgamma)}(\brho).
\end{align*}
\end{proof}

\begin{lemma}\label{lema42prima}
Let $(M_n)_{n\ge0}$ be a sequence of positive real numbers with $M_0=1$, which satisfies the (mg) condition and
let $f\in\mathcal{O}(D_R^N)$. For every $0<R'<R$ there exists $0<r<R'$ such that
$$\sup_{\bz\in D^N_{r}}\left\|f(\bz)\right\|_{\brho}\ll 2^NC\Theta^{(\boldsymbol{0})}(h\brho)$$
for $C=\sup_{\bz\in D^N_{R'}}|f(\bz)|$ and some $h>0$.
\end{lemma}
\begin{proof}
Let $f\in\mathcal{O}(D_R^N)$ and $0<R'<R$. By Cauchy's integral formula $f(\bz)\ll C\sum_{\balpha\in\N_0^N}\frac{\bz^{\balpha}}{R'^{|\balpha|}}$ with $C=\sup_{\bz\in D^N_{R'}}|f(\bz)|$. In view of Proposition~\ref{prop1} there exist $h,B'>0$ such that
$$\partial_{\bm,\bz}^{\bbeta}f(\bz)\ll C h^{|\bbeta|} M_{\bbeta}^{\bs}\sum_{\balpha\in\N_0^N}(B'\bz)^{\balpha}\quad\textrm{for every}\quad\bbeta\in\N_0^N.$$
Therefore $|\partial_{\bm,\bz}^{\bbeta}f(\bz)|\le 2^NCh^{|\bbeta|} M_{\bbeta}^{\bs}$ for $\bz\in D^N_{r}$ and $r=\frac{1}{2B'}$. We conclude that
\begin{align*}
\sup_{\bz\in D^N_{r}}\left\|f(\bz)\right\|_{\brho}&\ll\sum_{\balpha\in\N_0^N}\sup_{\bz\in D^N_{r}}|\partial_{\bm,\bz}^{\balpha}f(\bz)|\frac{{\brho}^{\balpha}}{M_{\balpha}^{\bs}}\\
&\ll\sum_{\balpha\ge\boldsymbol{0}}\frac{2^NCh^{|\balpha|} M_{\balpha}^{\bs}}{M_{\balpha}^{\bs}}\brho^{\balpha}=2^NC\Theta^{(\boldsymbol{0})}(h\brho).
\end{align*}
\end{proof}

\begin{lemma}\label{lema43prima}
Let $R>0$ and $f\in\mathcal{O}(D_R^N)$ such that 
$$\sup_{\bz\in D_r^N}\left\|f(\bz)\right\|_{\brho}\ll C\Theta^{(\bgamma)}(h\brho)$$
for some $C,h>0$, $0<r<R$ and $\bgamma\in\N^N_0$. Then for all $\bbeta\in\N_0^N$ we have
$$\sup_{\bz\in D_r^N}\left\|\partial_{\bm,\bz}^{\bbeta} f(\bz)\right\|_{\brho}\ll C h^{|\bbeta|}\Theta^{(\bgamma+\bbeta)}(h\brho).$$
\end{lemma}
\begin{proof}
Since $\sup_{\bz\in D^N_r}\left\|f(\bz)\right\|_{\brho}\ll C\Theta^{(\bgamma)}(h\brho)$, for any $\balpha\in\N_0^N$ it follows that
$$\sup_{\bz\in D^N_r}\frac{|\partial_{\bm,\bz}^{\balpha}f(\bz)|}{M_{\balpha}^{\bs}}\le \frac{C h^{|\balpha|} M_{\balpha+\bgamma}^{\bs}}{M_{\balpha}^{\bs}},$$
leading to $\sup_{\bz\in D^N_r}|\partial_{\bm,\bz}^{\balpha} f(\bz)|\le C h^{|\balpha|}M_{\balpha+\bgamma}^{\bs}$. Hence
$$\sup_{\bz\in D^N_r}\frac{|\partial_{\bm,\bz}^{\balpha+\bbeta}f(\bz)|}{M_{\balpha}^{\bs}}\le \frac{Ch^{|\balpha+\bbeta|}M_{\balpha+\bbeta+\bgamma}^{\bs}}{M_{\balpha}^{\bs}},$$
and therefore
$$\sup_{\bz\in D_r^N}\left\|\partial_{\bm,\bz}^{\bbeta}f(\bz)\right\|_{\brho}\ll C h^{|\bbeta|} \Theta^{(\bgamma+\bbeta)}(h\brho).$$
\end{proof}

A generalized version of \cite[Lemma 4.4]{michaliksuwinska} can be obtained by means of Lemma~\ref{lema42prima}.

\begin{lemma}\label{lema44prima}
Let $\overline{s}\in\R$, and let $(M_n)_{n\ge0}$ be a sequence of positive real numbers with $M_0=1$, which satisfies (mg). Let $f(t,\bz)=\sum_{n\ge0}f_n(\bz)t^n\in\mathcal{O}(D_{R}^N)[[t]]_{\overline{s}}^{(M_n)}$. Then there exist $0<r<R$ and constants $A,B,h>0$ such that
$$\sup_{\bz\in D_r^N}\left\|f_n(\bz)\right\|_{\brho}\ll A B^n\Theta^{(\boldsymbol{0})}(h\brho)M_n^{\overline{s}}$$
for all $n\in\N_0$.
\end{lemma}

\section{Moment equations with time variable coefficients}\label{secmom}
Let $(M_n)_{n\ge0}$ be a non-decreasing sequence of positive real numbers such that $M_0=1$, which satisfies (lc) and (mg) conditions. We also fix $(M_n)$-sequences of orders $s_0,s_1,\dots,s_N\geq 0$, respectively, say $(m_0(n))_{n\ge0}$, $(m_1(n))_{n\ge0},\ldots,(m_N(n))_{n\ge0}$, and we denote $\bm=(\bm(n))_{n\ge0}=((m_1(n))_{n\ge0},\dots,(m_N(n))_{n\ge0})$ and $\bs=(s_1,\dots,s_N)$. Let $\Lambda$ be a finite subset of $\N_0\times\N_0^N$, and for all $(j,\balpha)\in \Lambda$ let us take $a_{j\balpha}(t)\in\C[[t]]$. Let $K\ge 1$ and put
\begin{equation}\label{e344}
P(\partial_{m_0,t},\partial_{\bm,\bz})=\partial_{m_0,t}^{K}+\sum_{(j,\balpha)\in \Lambda}a_{j\balpha}(t)\partial_{m_0,t}^j\partial_{\bm,\bz}^{\balpha}.
\end{equation}

\begin{defin}
We define the \emph{Newton polygon} associated with (\ref{e344}) as the convex hull of the set 
$$\Delta(Ks_0,-K)\cup\bigcup_{(j,\balpha)\in\Lambda}\Delta(js_0+\balpha\cdot \bs,\hbox{ord}_t(a_{j\balpha})-j)$$
and denote it by $N(P,s_0,\bs)$.
Here $\Delta(a,b):=\{(x,y)\in\R^2\colon x\le a,\ y\ge b\}$ for all $(a,b)\in\R^2$ and
$\hbox{ord}_t(a_{j\balpha})$ denotes the order of zero of the function $a_{j\balpha}(t)$ at $t=0$.
\end{defin}

We assume that the following condition holds:
\begin{itemize}
\item[(a)] $\hbox{ord}_t(a_{j\balpha}) \ge \max\{0,j-K+1\}$ for all $(j, \balpha) \in\Lambda$. 
\end{itemize}

Regarding (a), we have that the first positive slope of $N(P,s_0,\bs)$, which will be denoted by $k_1$, is determined by
$$\frac{1}{k_1}=\max\left\{0,\max_{(j,\balpha)\in\Lambda}\left\{\frac{s_0(j-K)+\bs\cdot\balpha}{q_{j\balpha}}\right\}\right\},$$
where $q_{j\balpha}=\hbox{ord}_t(a_{j\balpha})-j+K$.

The main problem under study is the Cauchy problem 

\begin{equation} 
\left\{ \begin{aligned}
P(\partial_{m_0,t}, \partial_{\bm, \bz}) u(t,\bz)&=f(t,\bz)\\
\partial_{m_0, t}^j 
u(0,\bz)&=\varphi_j(\bz), \quad j=0,\ldots,K-1.
\end{aligned}
   \right.
	\label{epral} 
\end{equation} 

Additionally, we assume:
\begin{enumerate} 
\item[(b)] $\varphi_j\in\mathcal{O}(D^N_R)$ for all $j=0,1,\ldots, K-1$ and $f\in\mathcal{O}(D^N_R)[[t]]^{(M_n)}_{1/k_1}$ for some $R>0$, 
\item[(c)] $a_{j\balpha}(t)\in\C[[t]]_{1/k_1}^{(M_n)}$ for all $(j, \balpha) \in\Lambda$, 
\item[(d)] $m_0$ is a regular $(M_n)$-sequence. 
\end{enumerate} 

In the next main lemma we estimate $\bm$-formal norms of the coefficients of the formal solution of (\ref{epral}) under  conditions listed above. 
\begin{lemma}\label{lema51prima}
Fix $\bbeta\in \N^N_0$ and $\nu>0$ satisfying $\bs\cdot\bbeta=s_0\nu+\frac{1}{k_1}$. Let $\hat{u}(t, \bz) =\sum_{n\ge0} u_n(\bz) t^n$ be a formal solution of (\ref{epral}). Then, there exist $r\in(0,R)$ and constants $C, H, h>0$ such that
\begin{equation}\label{e333}
\left\|u_n(\bz)\right\|_{\brho}\ll \frac{CH^n}{M_n^{\nu s_0} }\Theta^{(n\bbeta)}(h \brho)
\end{equation}
 for all $n\in\N_0$ and $\bz\in D^N_r$.
\end{lemma} 
\begin{proof} 
Analogous estimates as those in \cite[Lemma 5.1]{michaliksuwinska} yield that for $n\ge K$ one has 
$$u_n(\bz)=\frac{m_0(n-K)}{m_0(n)}\left(f_n(\bz)-\sum_{(j, \balpha) \in \Lambda}\sum_{p=q_{j\balpha}}^n c_{j\balpha p}\frac{m_0(n-p)} {m_0(n-p-j)}\partial_{\bm,\bz}^{\balpha} u_{n-p}(\bz) \right),$$
where $t^Kf(t,\bz)=\sum_{n\ge K} f_n(\bz) t^n$ and $t^{K-j}a_{j\balpha}(t)=\sum_{p\ge q_{j\balpha}} c_{j\balpha p} t^p,$ for all $(j, \balpha)\in \Lambda$.

Since $a_{j\balpha}(t) \in\C[[t]]_{1/k_1}^{(M_n)}$, for every $(j, \balpha)\in\Lambda$ there exist $A_{j\balpha}, B>0$ such that
$$|c_{j\balpha p}|\le A_{j\balpha}B^pM^{1/k_1}_{p-q_{j\balpha}}$$
for every $p\ge q_{j\balpha}$. Moreover, Lemma~\ref{lema44prima} yields the existence of $\tilde{A},\tilde{B}, h>0$ such that 
$$\left\|f_n(\bz)\right\|_{\brho}\ll\tilde{A}\tilde{B}^n\Theta^{(\boldsymbol{0})}(h\brho)M_n^{1/k_1}, $$
for every $n\in\N$ and $\bz\in D^N_r$. Therefore, one arrives to the following estimates for $n\ge K$:
$$\left\|u_n(\bz)\right\|_{\brho}\ll I+II,$$
where
$$I=\left(\frac{M_{n-K}}{M_n}\right)^{s_0}\frac{1}{\mathfrak{a}^{K}}\tilde{A}\tilde{B}^n\Theta^{(\boldsymbol{0})}(h\brho) (M_n)^{1/k_1}$$
and
\begin{equation*}
II=\left(\frac{M_{n-K}}{M_n}\right)^{s_0}\frac{1}{\mathfrak{a}^{K}}\left[\sum_{(j,\balpha)\in\Lambda}\sum_{p=q_{j\balpha}}^{n}A_{j\balpha}B^pM_{p-q_{j\balpha}}^{1/k_1} \mathfrak{A}^j\left(\frac{M_{n-p}}{M_{n-p-j}}\right)^{s_0}\left\|\partial_{m,\bz}^{\balpha}u_{n-p}(\bz)\right\|_{\brho}\right].
\end{equation*}
We provide upper bounds for $I$ and $II$ in two steps, and obtain (\ref{e333}) by induction on $n$.

Let $\tilde{C}$ be the constant obtained after the application of Lemma~\ref{lema4}. In the following upper estimates we apply that result and also Lemma~\ref{lema41prima} and Lemma~\ref{lema2}, and monotonicity of $(M_n)_{n\ge0}$. We have
\begin{align}
I&\ll \tilde{C}^{Ks_0}\left(\frac{M_{n-1}}{M_n}\right)^{Ks_0}\frac{1}{\mathfrak{a}^{K}}\tilde{A}\tilde{B}^{n}\Theta^{(\boldsymbol{0})}(h\brho)(M_n)^{1/k_1}\nonumber\\
&\ll\left(\frac{\tilde{C}^{s_0}}{\mathfrak{a}}\right)^{K}\tilde{A}\tilde{B}^n\left(\frac{M_{n-1}}{M_n}\right)^{Ks_0}(M_{n})^{1/k_1}\frac{\Theta^{(n\bbeta)}(h\brho)}{M_{n\bbeta}^{\bs}}\nonumber\\
&\ll \left(\frac{\tilde{C}^{s_0}}{\mathfrak{a}}\right)^{K}\tilde{A}\tilde{B}^n\frac{\Theta^{(n\bbeta)}(h\brho)}{M_{n}^{s_0\nu}}\ll \frac{1}{2}CH^n\frac{\Theta^{(n\bbeta)}(h\brho)}{M_n^{s_0\nu}},\label{e404}
\end{align}
for $C\ge 2(\tilde{C}^{s_0}/\mathfrak{a})^{K}\tilde{A}$ and $H\ge \tilde{B}$.

On the other hand, Lemma~\ref{lema43prima} together with the inductive hypothesis yield
$$\left\|\partial_{\bm,\bz}^{\balpha}u_{n-p}(\bz)\right\|_{\brho}\ll \frac{CH^{n-p}}{M_{n-p}^{\bnu s_0}} h^{|\balpha|}\Theta^{((n-p)\bbeta+\balpha)}(h\brho),$$
which allows to apply Lemma~\ref{lema41prima} and Lemma~\ref{lema4} twice to arrive at
\begin{align*}
II\ll{}&\left(\frac{\tilde{C}^{s_0}}{\mathfrak{a}}\right)^{K}\left(\frac{M_{n-1}}{M_{n}}\right)^{Ks_0}\sum_{(j,\alpha)\in\Lambda}\sum_{p=q_{j\balpha}}^{n}A_{j\balpha}\mathfrak{A}^jB^p(M_{p-q_{j\balpha}})^{1/k_1}\left(\frac{M_{n-p}}{M_{n-p-j}}\right)^{s_0}\\
&\times \frac{CH^nh^{|\balpha|}}{(M_{n-p})^{\nu s_0}H^p} \frac{M^s_{(n-p)\bbeta+\balpha}}{M^{\bs}_{n\bbeta}}\Theta^{(n\bbeta)}(h\brho)\\
\ll{}&\left(\frac{\tilde{C}^{s_0}}{\mathfrak{a}}\right)^{K}\frac{CH^n}{(M_{n})^{\nu s_0}}\left(\frac{M_{n-1}}{M_{n}}\right)^{Ks_0}\sum_{(j,\balpha)\in\Lambda}\sum_{p=q_{j\balpha}}^{n}\frac{h^{|\balpha|} \mathfrak{A}^jA_{j\balpha}B^p}{H^p}(M_{p-q_{j\balpha}})^{1/k_1}\left(\frac{M_n}{M_{n-1}}\right)^{j s_0}\\
&\times \left(\frac{M_{n}}{M_{n-p}}\right)^{\nu s_0}\frac{M_{(n-p)\bbeta+\balpha}^{\bs}}{M_{n\bbeta}^{\bs}}\Theta^{(n\bbeta)}(h\brho)\\
\ll{}&\left(\frac{\tilde{C}^{s_0}}{\mathfrak{a}}\right)^{K}\frac{CH^n}{(M_{n})^{\nu s_0}}\sum_{(j,\balpha)\in\Lambda}\sum_{p=q_{j\balpha}}^{n}\frac{h^{|\balpha|} A_{j\balpha}\mathfrak{A}^jB^p}{H^p}\left(\frac{M_{n-1}}{M_{n}}\right)^{Ks_0}\\
&\times\left(\frac{M_n}{M_{n-1}}\right)^{\frac{p-q_{j\balpha}}{k_1}+js_0+p\nu s_0}\tilde{C}^{(p\bbeta-\balpha)\cdot \bs}\left(\frac{M_{n-1}}{M_n}\right)^{(p\bbeta-\balpha)\cdot \bs}\Theta^{(n\bbeta)}(h\brho)\\
\ll{}&\frac{CH^n}{(M_{n})^{\nu s_0}}\sum_{(j,\balpha)\in\Lambda}\sum_{p=q_{j\balpha}}^{n}\left(\frac{\tilde{C}^{s_0}}{\mathfrak{a}}\right)^{K} h^{|\balpha|} A_{j\balpha}\mathfrak{A}^j\tilde{C}^{-\balpha\cdot \bs}\left(\frac{\tilde{C}^{\bbeta\cdot \bs}B}{H}\right)^p\left(\frac{M_n}{M_{n-1}}\right)^{\Phi}\Theta^{(n\bbeta)}(h\brho),
\end{align*}
with 
\begin{multline*}
\Phi=\frac{p-q_{j\alpha}}{k_1}+j s_0+p\nu s_0-K s_0-p\beta\cdot s+\alpha\cdot s=\frac{p-q_{j\alpha}}{k_1}+j s_0+p\nu s_0-K s_0-p s_0\nu-\frac{p}{k_1}+\alpha\cdot s\\
=-\frac{q_{j\alpha}}{k_1}+s_0(j-K)+\alpha\cdot s\le 0,
\end{multline*}
from the definition of $k_1$. We conclude that
\begin{align}
II&\ll \frac{CH^n}{(M_{n})^{\nu s_0}}\sum_{(j,\balpha)\in\Lambda}\left(\frac{\tilde{C}^{s_0}}{\mathfrak{a}}\right)^{K} h^{|\balpha|} A_{j\balpha}\mathfrak{A}^j\tilde{C}^{-\balpha s}\left( \sum_{p\ge q_{j\balpha}}\left(\frac{\tilde{C}^{\bbeta\cdot \bs}B}{H}\right)^p\right)\Theta^{(n\bbeta)}(h\brho)\nonumber\\
&\ll\frac{CH^n}{(M_{n})^{\nu s_0}}\Theta^{(n\bbeta)}(h\brho)\sum_{(j,\balpha)\in\Lambda}\left(\frac{\tilde{C}^{s_0}}{\mathfrak{a}}\right)^{K} h^{|\balpha|} A_{j\balpha}\mathfrak{A}^j\tilde{C}^{-\balpha \bs} \frac{\left(\frac{\tilde{C}^{\bbeta\cdot \bs}B}{H}\right)^{q_{j\balpha}}}{1-\frac{\tilde{C}^{\bbeta\cdot \bs}B}{H}}\nonumber\\
&\ll \frac{1}{2}\frac{CH^n}{(M_{n})^{\nu s_0}}\Theta^{(n\bbeta)}(h\brho),\label{e429}
\end{align}
for large enough $H$. The result follows from (\ref{e404}) and (\ref{e429}).
\end{proof}

Now we are ready to state one of the main results of the paper.
\begin{theo}\label{theo1}
Let $\hat{u}(t,\bz)=\sum_{n\ge0}u_n(\bz)t^n$ be a formal solution of (\ref{epral}). Then $\hat{u}(t,\bz)\in\mathcal{O}(D^N_r)[[t]]_{1/k_1}^{(M_n)}$ for some $r\in(0,R)$, i.e., for all $r'\in(0,r)$ there exist $\tilde{C},\tilde{H}>0$ such that
$$\sup_{\bz\in D^N_{r'}}|u_n(\bz)|\le\tilde{C}\tilde{H}^n(M_n)^{1/k_1},\quad n\in\N_0.$$
\end{theo}
\begin{proof}
In view of Lemma~\ref{lema51prima} and Lemma~\ref{lema2prima} we get for all $n\in\N_0$ that
\begin{align*}
|u_n(t)|&=\left\|u_n(t)\right\|_{\boldsymbol{0}}\le\frac{C H^n}{(M_n)^{\nu s_0}}\Theta^{(n\bbeta)}(\boldsymbol{0})=\frac{C H^n}{(M_n)^{\nu s_0}}M_{n\bbeta}^{\bs}\le \frac{C C'^{nN}H^n}{(M_n)^{\nu s_0}}M_n^{\bbeta\cdot {\bs}}\\
&=CC'^nH^n(M_n)^{1/k_1}=\tilde{C}\tilde{H}^n(M_n)^{1/k_1},
\end{align*}
which leads to the result.
\end{proof}

\section{Moment equations with constant coefficients --- two-dimensional case}\label{sectwo}
In the special case of the equations with constant coefficients we attain an analogous result as in the previous section, with the assumption that $(M_n)_{n\ge 0}$ is a (lc)-sequence replaced by the weaker condition ($\star$). It is also worth emphasizing that in this special case we do not need to assume that the sequence $m_0=(m_0(n))_{n\ge0}$ is regular.

First, we consider equations in two complex variables $(t,z)\in\C^2$.

Let $(M_n)_{n\ge0}$ be a sequence of positive real numbers with $M_0=1$. Let $m_0=(m_0(n))_{n\ge0}$ and $m_1=(m_1(n))_{n\ge0}$ be two $(M_n)$-sequences of orders $s_0,s_1\ge0$, respectively.
We consider the Cauchy problem
\begin{equation}\label{e5}
\left\{ \begin{aligned}
             P(\partial_{m_0,t},\partial_{m_1,z})u(t,z)&=\hat{f}(t,z)\\
						 \partial_{m_0,t}^ju(0,z)&=0,\quad j=0,\ldots,M-1,
             \end{aligned}
   \right.
\end{equation}
with
\begin{equation}\label{e6}
P(\partial_{m_0,t},\partial_{m_1,z})=\partial_{m_0,t}^{K}-\sum_{j=1}^KP_j(\partial_{m_1,z})\partial_{m_0,t}^{K-j}
\end{equation}

As before, we denote by $k_1$ the first positive slope of the Newton polygon associated with $P$.

\begin{theo}\label{theo2}
Let $(M_n)_{n\ge0}$ be an increasing sequence of positive real numbers with $M_0=1$, which satisfies (mg) and ($\star$). Let $\hat{f}(t,z)$ be of $(M_n)$-order $1/k_1>0$. Then the formal solution $\hat{u}(t,z)$ of (\ref{e5}) is of $(M_n)$-order $1/k_1$.
\end{theo}
\begin{proof}
Let us write $\hat{f}(t,z)=\sum_{n\ge0}\frac{f_n(z)}{m_0(n)}t^n$. Then the formal solution of (\ref{e5}) is given by (see \cite[Proposition 6]{M}) 
$$\hat{u}(t,z)=\sum_{n\ge0}(\partial_{m_0,t}^{-1})^{n+1}g_n(\partial_{m_1,z})\hat{f}(t,z)=\sum_{\iota\ge0}\frac{t^\iota}{m_1(\iota)}\sum_{n=0}^{\iota-1}g_n(\partial_{m_1,z})f_{\iota-n-1}(z),$$
where $g_n(\zeta)$ is a polynomial satisfying the difference equation
$$g_n(\zeta)=\sum_{j=1}^{K}P_j(\zeta)g_{n-j}(\zeta),\quad n\ge K,$$
and 
$$g_0(\zeta)\equiv\ldots\equiv g_{K-2}(\zeta)\equiv 0,\quad g_{K-1}(\zeta)\equiv 1.$$
Let us denote $u_\iota(z):=\sum_{n=0}^{\iota-1}g_n(\partial_{m_1,z})f_{\iota-n-1}(z)$. We have $\hat{u}(t,z)=\sum_{\iota\ge0}\frac{t^\iota}{m_0(\iota)}u_{\iota}(z)$.

The classical theory of difference equations determines that 
\begin{equation}\label{e7}
g_n(\zeta)=\sum_{\alpha=1}^{k}\sum_{\beta=1}^{n_\alpha}c_{\alpha\beta}(\zeta)\frac{n!}{(n+1-\beta)!}\ell_{\alpha}^n(\zeta),
\end{equation}
where $c_{\alpha\beta}(\zeta)$ are algebraic functions, and $\ell_1(\zeta),\ldots,\ell_k(\zeta)$ are roots of the characteristic equation $P(\ell,\zeta)=0$ with multiplicities $n_1,\ldots,n_k$, respectively. The fact that $\ell_\alpha(\zeta)$ is an algebraic function for every $\alpha=1,\ldots,k$ entails the existence of $\ell_\alpha\in\C^{\star}$ and $q_\alpha\in\mathbb{Q}_{+}$ such that
$$\lim_{\zeta\to\infty}\frac{\ell_\alpha(\zeta)}{\zeta^{q_{\alpha}}}=\ell_\alpha.$$
The definition of the Newton polygon associated to $P$, and $k_1$ yield
$$k_1=\frac{1}{q^{\star}s_1-s_0},\hbox{ where } q^{\star}=\max\{q_1,\ldots,q_k\}.$$
In view of (\ref{e7}) there exist $a\in\N$ and $A,B>0$ such that one can write
$$g_n(\zeta)=\sum_{\alpha=0}^{\lfloor q^\star n+a\rfloor}a_{\alpha n}\zeta^\alpha,$$
where the coefficients $a_{\alpha n}$ satisfy
$$
\sum_{\alpha=0}^{\lfloor q^\star n+a\rfloor}|a_{\alpha n}|\leq A B^n\quad\textrm{for every}\quad n\in\N_0.
$$

Given a convergent power series $\varphi(z)\ll C\sum_{n=0}^{\infty} (Dz)^n$, for some $C,D>0$, by Proposition~\ref{prop1} and the assumption that $(M_n)_{n\ge0}$ is an increasing sequence, one has
\begin{multline*}
q_n(\partial_{m_1,z})\varphi(z)=\sum_{\alpha\le q^{\star}n+a}a_{\alpha n}\partial_{m_1,z}^{\alpha}\varphi(z)\ll \sum_{\alpha\le q^{\star}n+a}|a_{\alpha n}|CH^{\alpha}M_\alpha^{s_1}\sum_{n\ge0}(D'z)^k\\
\ll \big(\sum_{\alpha\le q^{\star}n+a}|a_{\alpha n}|\big) C H^{q^{\star}n+a}M_{\lfloor q^\star n+a\rfloor}^{s_1}\sum_{k\ge 0}(D'z)^k\ll C\tilde{C} \tilde{H}^n M_{\lfloor q^{\star}n\rfloor}^{s_1}\sum_{k\ge0}(D'z)^k.
\end{multline*}
Since $q^\star\in\mathbb{Q}_+$, by Lemma \ref{le:11}
$$q_n(\partial_{m_1,z})\varphi(z)\ll C\tilde{C}_2\tilde{H}_2^n M_n^{q^{\star}s_1}\sum_{k\ge0}(D'z)^k,\quad n\ge 0,$$
for some $\tilde{C}_2,\tilde{H}_2>0$.

Since $\hat{f}(t,z)$ is of $(M_n)$-order $1/k_1$, there exist $A,B,D>0$ such that
$$f_n(z)\ll A B^n (M_n)^{1/k_1}\sum_{k\ge0}(Dz)^k,$$
in a common neighborhood of the origin for all $n\in\N_0$. Hence, using ($\star$) and
the monotonicity of $(M_n)_{n\ge0}$, we get
\begin{align*} 
u_{\iota}(z)={}&\sum_{n=0}^{\iota-1}q_n(\partial_{m_1,z})f_{\iota-n-1}(z)\ll \sum_{n=0}^{\iota-1}A B^{\iota-n-1}(M_{\iota-n-1})^{\frac{1}{k_1}}\tilde{C}_2(\tilde{H}_2)^nM_n^{q^{\star}s_1}\sum_{k\ge0}(D'z)^k\\
\ll{}& A'(B')^{\iota}\left(\sum_{n=0}^{\iota-1}(M_{\iota-n-1})^{\frac{1}{k_1}}(M_n)^{q^{\star}s_1}\right)\sum_{k\ge 0}(D' z)^k\\
\ll{}& A'(B')^{\iota}\left(\sum_{n=0}^{\iota-1}(M_{\iota-n-1})^{\frac{1}{k_1}}(M_n)^{\frac{1}{k_1}+s_0}\right)\sum_{k\ge 0}(D' z)^k\\
\ll{}& A' (B'_1)^{\iota}\left(\sum_{n=0}^{\iota-1}(M_{\iota-1})^{\frac{1}{k_1}}(M_n)^{s_0}\right)\sum_{k\ge 0}(D' z)^k\ll A'_2 (B'_2)^{\iota} (M_{\iota})^{\frac{1}{k_1}}(M_\iota)^{s_0}\left(\sum_{k\ge0}(D'z)^k\right),
\end{align*}
for some $A'_2,B'_2>0$. This entails the existence of $A,B>0$ and $r>0$ such that
$$\frac{u_\iota(z)}{m_1(\iota)}\le A B^\iota (M_\iota)^{\frac{1}{k_1}},\quad \iota\ge0,\quad z\in D_r.$$
\end{proof}

\section{Moment equations with constant coefficients --- multidimensional case}\label{secsev}
The multidimensional case $(t,\bz)\in\C^{1+N}$ is more complex and needs some more effort. In particular, our reasoning rests on Lemma~\ref{le:12}, which only holds when certain parameters are non-negative rational numbers $p/q$. Therefore, we achieve the main result under the assumption that the orders $s_1,\dots,s_N$ are rational numbers.

More precisely, we consider the Cauchy problem for linear generalized moment-PDE equations with constant coefficients
\begin{equation}\label{e8}
\left\{ \begin{aligned}
             P(\partial_{m_0,t},\partial_{\bm,\bz})u(t,\bz)&=\hat{f}(t,\bz)\\
						 \partial_{m_0,t}^ju(0,\bz)&=0,\quad j=0,\ldots,K-1,
             \end{aligned}
   \right.
\end{equation}
with $t\in\C$, $\bz=(z_1,\ldots,z_N)\in\C^N$ for some $N\ge 1$, and where $m_0,m_1,\ldots,m_N$ are $(M_n)$-sequences of orders $s_0,s_1,\ldots,s_N\ge 0$, respectively. We write $\bs=(s_1,\ldots,s_N)$, $\bm=(m_1,\ldots,m_N)$, and denote $\partial_{\bm,\bz}^{\balpha}=\partial_{m_1,z_1}^{\alpha_1}\cdots \partial_{m_N,z_N}^{\alpha_N}$, for every $\balpha=(\alpha_1,\ldots,\alpha_N)\in\N_0^{N}$. We put 
\begin{equation}\label{e9}
P(\partial_{m_0,t},\partial_{\bm,\bz})=\partial_{m_0,t}^{K}-\sum_{(j,\balpha)\in\Lambda}a_{j,\balpha}\partial_{m_0,t}^{j}\partial_{\bm,\bz}^{\balpha},
\end{equation}
where $\Lambda\subseteq\{0,\ldots,K-1\}\times\N_0^{N}$ is a finite subset of indices.

\begin{theo}\label{prop4}
Let $(M_n)_{n\ge0}$ be an increasing sequence of positive real numbers with $M_0=1$, which satisfies (mg) and ($\star$). Let $\hat{f}(t,\bz)$ be of $(M_n)$-order $1/k_1>0$. Additionally we assume that $s_1,\dots,s_N\in\overline{\Q}_+$. Then, the formal solution $\hat{u}(t,\bz)$ of (\ref{e8}) is of $(M_n)$-order $1/k_1$.
\end{theo}
\begin{proof}
Let us write $\hat{f}(t,\bz)=\sum_{n\ge0}\frac{f_n(\bz)}{m_0(n)}t^n$. As in the previous section, the formal solution of (\ref{e8}) can be written in the form
$$\hat{u}(t,\bz)=\sum_{n\ge0}(\partial_{m_0,t}^{-1})^{n+1}g_n(\partial_{\bm,\bz})\hat{f}(t,\bz)=\sum_{\iota\ge0}\frac{t^\iota}{m_0(\iota)}\sum_{n=0}^{\iota-1}g_n(\partial_{\bm,\bz})f_{\iota-n-1}(\bz),$$
where $g_n(\bzeta)=g_n(\zeta_1,\ldots,\zeta_N)$ is a polynomial satisfying the difference equation
$$g_n(\bzeta)=\sum_{j=1}^{K}P_j(\bzeta)g_{n-j}(\bzeta),\quad n\ge K,$$
and 
$$g_0(\bzeta)\equiv\ldots\equiv g_{K-2}(\bzeta)\equiv 0,\quad g_{K-1}(\bzeta)\equiv 1.$$
Let us denote $u_\iota(\bz):=\sum_{n=0}^{\iota-1}g_n(\partial_{\bm,\bz})f_{\iota-n-1}(\bz)$. We have $\hat{u}(t,\bz)=\sum_{\iota\ge0}\frac{t^\iota}{m_0(\iota)}u_{\iota}(\bz)$.

The classical theory of difference equations determines that 
\begin{equation}\label{e11}
g_n(\bzeta)=\sum_{l=1}^{k}\sum_{m=1}^{n_l}c_{lm}(\bzeta)\frac{n!}{(n+1-m)!}\ell_{l}^n(\bzeta),
\end{equation}
where $c_{lm}(\bzeta)$ are algebraic functions, and $\ell_1(\bzeta),\ldots,\ell_k(\bzeta)$ are roots of the characteristic equation $P(\ell,\bzeta)=0$ with multiplicities $n_1,\ldots,n_k$, respectively. The function $\ell_l(\bzeta)=\ell_l(\zeta_1,\ldots,\zeta_N)$ is an algebraic function for every $l=1,\ldots,k$, and $\bzeta=(\zeta_1,\ldots,\zeta_N)\in\C^N$. We write $\overline{\ell}_l(\xi)=\ell_l(\xi^{s_1},\ldots,\xi^{s_N})$ for $\xi\in\C$. Then, $\overline{\ell}_l(\xi)$ is a holomorphic function for sufficiently large $|\xi|$ with a moderate growth at infinity, i.e.
$$\lim_{\xi\to\infty}\frac{\overline{\ell}_l(\xi)}{\zeta^{\overline{q}_{l}}}=\ell_l,$$
for some $\ell_l\in\C\setminus\{0\} $ and $\overline{q}_l\in\overline{\mathbb{R}}_{+}$.
The first positive slope of the Newton polygon is given by $k_1=(\overline{q}^{\star}-s_0)^{-1},$ where $\overline{q}^{\star}=\max\{\overline{q}_1,\ldots,\overline{q}_k\}$. 

In view of (\ref{e11}) there exists $a\in\N$ such that one can write every polynomial $g_n(\bzeta)$ as
$$g_n(\bzeta)=\sum_{\balpha\in\N_0^N\atop \balpha\cdot \bs\le \overline{q}^{\star}n+a}a_{n\balpha}\bzeta^{\balpha}$$
and there exist $A,B>0$ such that
$$\sum_{\balpha\in\N_0^N\atop \balpha\cdot \bs\le \overline{q}^{\star}n+a}|a_{n\balpha}|\leq AB^n\quad\textrm{for every}\quad n\in\N_0.$$

Given a convergent power series $\varphi(\bz)\ll C\sum_{\bbeta\in\N_0^N} (D\bz)^{\bbeta}$, for some $C,D>0$, by Proposition~\ref{prop1}, Lemma \ref{le:12} and the assumption that $(M_n)_{n\ge0}$ is an increasing sequence satisfying (mg) and ($\star$), one has
\begin{multline*}
q_n(\partial_{\bm,\bz})\varphi(\bz)=\sum_{\balpha\in\N_0^N\atop\balpha\cdot \bs\le \overline{q}^{\star}n+a}a_{n\balpha}\partial_{\bm,\bz}^{\balpha}\varphi(\bz)\ll \sum_{\balpha\in\N_0^N\atop \balpha\cdot \bs\le \overline{q}^{\star}n+a}|a_{n\balpha}|CH^{|\balpha|}M_{\balpha}^{\bs}\sum_{\bbeta\in\N_0^N}(D'\bz)^{\bbeta}\\
\ll \sum_{\balpha\in\N_0^N\atop \balpha\cdot \bs\le \overline{q}^{\star}n+a}|a_{n\balpha}|C\tilde{C}\tilde{H}^{|\balpha|}M_{\alpha_1 s_1}\cdots M_{\alpha_N s_N}
\sum_{\bbeta\in\N_0^N}(D'\bz)^{\bbeta}\\
\ll \big(\sum_{\balpha\in\N_0^N\atop \balpha\cdot \bs\le \overline{q}^{\star}n+a}|a_{n\balpha}|\big)C\bar{C} \bar{H}^{q^{\star}n+a}M_{\lfloor q^\star n+a\rfloor}\sum_{\bbeta\in\N_0^N}(D'\bz)^{\bbeta}\ll C \tilde{C} \tilde{H}^n M_{n}^{\overline{q}^{\star}}\sum_{\bbeta\in\N_0^N}(D'\bz)^{\bbeta}.
\end{multline*}
Following analogous steps as those at the end of Theorem \ref{theo2}, we complete the proof. 
\end{proof}

\section{Some technical auxiliary lemmas on sequences of numbers}\label{secauxlemas}
In the last section we collect the auxiliary lemmas on sequences of numbers
$(M_n)_{n\ge0}$ which are used in the paper.

\begin{lemma}
Let $(M_n)_{n\ge0}$ be an (lc) sequence. For every $n\in\N$ and $p\in\N_0$ with $p\le n$ one has
$$\frac{M_n}{M_{n-p}}\le\left(\frac{M_n}{M_{n-1}}\right)^p.$$
\end{lemma}
\begin{proof}
Since $(M_n)_{n\ge 0}$ is (lc), the sequence of quotients $(M_n/M_{n-1})_{n\ge 1}$ is increasing. Hence,
$$\frac{M_n}{M_{n-p}}=\frac{M_n}{M_{n-1}}\frac{M_{n-1}}{M_{n-2}}\cdots\frac{M_{n+1-p}}{M_{n-p}}\le\left(\frac{M_n}{M_{n-1}}\right)^p,$$
for all $n\in\N$ and $p\in\N_0$ with $p\le n$.
\end{proof}

\begin{lemma} \label{lema2}
Let $(M_n)_{n\ge0}$ be an (lc) sequence. For every $n,d\in\N_0$ one has $(M_n)^d\le M_{dn}$.
\end{lemma}
\begin{proof}
It is a direct consequence of the definition of (lc) sequences. Indeed, for all $n,d\in\N_0$ one has 
$$(M_n)^d =\left(\frac{M_n}{M_{n-1}}\right)^d\left(\frac{M_{n-1}}{M_{n-2}}\right)^d\cdots\left(\frac{M_1}{M_{0}}\right)^d\le \frac{M_{dn}}{M_{dn-1}}\frac{M_{dn-1}}{M_{dn-2}}\cdots\frac{M_1}{M_{0}}= M_{dn}.$$
\end{proof}

\begin{lemma}\label{lema2prima}
Let $(M_n)_{n\ge0}$ be an (mg) sequence. For every $d\in\N$ there exists $C'>0$ such that $M_{dn}\le C'^n(M_n)^{d}$ for all $n\in\N$.
\end{lemma}
\begin{proof}
It is a direct consequence of (mg) condition satisfied by $(M_n)_{n\ge0}$ applied recursively. We take $C'=B^{(d+1)d/2}$ where $B$ is the constant associated to the (mg) property of $(M_n)_{n\ge0}$.
\end{proof}

\begin{lemma}
Let $(M_n)_{n\ge0}$ be an (lc) sequence. For all $q,p\in\N$ with $p\le q$ we have
$$M_{p}\le \left(\frac{M_q}{M_{q-1}}\right)^{p}.$$
\end{lemma}
\begin{proof}
The proof is based on the monotonicity of the sequence of quotients $(M_n/M_{n-1})_{n\ge 1}$. For all $q,p\in\N$ with $p\le q$, one has
$$M_{p}=\frac{M_p}{M_{p-1}}\cdots \frac{M_1}{M_{0}}\le \frac{M_q}{M_{q-1}}\cdots^{(p)}\frac{M_q}{M_{q-1}}=\left(\frac{M_q}{M_{q-1}}\right)^{p}.$$
\end{proof}

\begin{lemma}\label{lema4}
Let $(M_n)_{n\ge0}$ be an (lc) and (mg) sequence. There exists $C>0$ such that 
\begin{equation}\label{e389}
\frac{M_{dn-a}}{M_{dn}}\le C^a\left(\frac{M_{n-1}}{M_n}\right)^{a},
\end{equation}
for all $d,n,a\in\N$.
\end{lemma}
\begin{proof}
From the very definition of (lc) and (mg) conditions we can guarantee the existence of $C>0$ such that
\begin{equation}\label{e388}
\left(\frac{M_p}{M_{p-1}}\right)^{p}\le \frac{M_{2 p}}{M_p}\le \frac{B^{2p}(M_p)^2}{M_p}=C^pM_p,
\end{equation}
for every $p\in\N$, and where $B$ is the constant involved in the (mg) condition. Observe that the previous statement can be rewritten in the form
$$\sup_{p\in\N}\frac{M_p/M_{p-1}}{(M_p)^{1/p}}<\infty.$$

In a first step, assume that $d=1$. Equation (\ref{e389}) reads as follows: 
\begin{equation}\label{e403}
\left(\frac{M_n}{M_{n-1}}\right)^a\le C^a\frac{M_n}{M_{n-a}},
\end{equation}
for all $n\in\N$, $n\geq a$. Equation (\ref{e403}) is satisfied for $a=0,1$, and also for $a=n$ in view of (\ref{e388}). For the rest of values, we apply an induction argument with respect to $a$. Let $(m_n)_{n\ge1}$ be the sequence of quotients $m_n:=M_n/M_{n-1}$, for $n\ge 1$, and assume that $(m_n)^a\le C^am_{n-a+1}\cdots m_n$.

Fix $a_0\in\{0,\ldots,n-1\}$ such that $C m_{n-a_0-1}\le m_n\le C m_{n-a_0}$ (with the notation $m_0:=0$). If $a\le a_0$ then $m_n\le C m_{n-a}$ and $(m_n)^{a+1}\le m_n C^a m_{n-a+1}\cdots m_n\le C^{a+1}m_{n-a}\cdots m_n$. Therefore, (\ref{e403}) holds for $a\le a_0$. On the other hand, if $a>a_0$, then
$$(m_n)^{a-1}\le C^a m_{n-a+1}\cdots m_n\frac{1}{m_n}\le C^{a-1}m_{n-a+2}\cdots m_n.$$
Since (\ref{e403}) holds for $a=n$, we have attained (\ref{e403}) for $a>a_0$. 

In the general case $d\ge 1$ one can apply (\ref{e403}) to arrive at 
$$\left(\frac{M_n}{M_{n-1}}\right)^a\le \left(\frac{M_{dn}}{M_{dn-1}}\right)^a\le C^a\frac{M_{dn}}{M_{dn-a}}.$$
\end{proof}

\begin{lemma}
 \label{le:11}
 Let $(M_n)_{n\ge0}$ be an (mg) sequence satisfying the property ($\star$). For every $p,q\in\N$ there exist positive constants $C,D$ such that
 $$
 M_{\lfloor np/q \rfloor}\leq C D^n M_n^{p/q}\quad \textrm{for every}\quad n\in\N_0.
 $$
\end{lemma}
\begin{proof}
 By Lemma \ref{lema2prima}
 $$M_{np}\leq B^{np(p+1)/2} M_n^p$$
 and by ($\star$)
 $$M^q_{\lfloor np/q \rfloor}\leq
 b^{-\lfloor np/q \rfloor q(q+1)/2}
 M_{q\lfloor np/q \rfloor}\leq \tilde{C}\tilde{D}^nM_{np}$$ for some $\tilde{C},\tilde{D}>0$.
 Hence finally
 $$
 M_{\lfloor np/q \rfloor}\leq C D^n M_n^{p/q}
 $$
 with $C=\tilde{C}^{1/q}$ and $D=\tilde{D}^{1/q}B^{p(p+1)/2q}$.
\end{proof}

\begin{lemma}
 \label{le:12}
 Let $(M_n)_{n\ge0}$ be an (mg) sequence satisfying property ($\star$). For every $p,q\in\N$ there exist $C,D>0$ such that
 $$
 M_n^{p/q}\leq C D^n M_{\lfloor np/q \rfloor}\quad \textrm{for every}\quad n\in\N_0.
 $$
\end{lemma}
\begin{proof}
 By ($\star$)
 $$M_n^{p}\leq b^{-np(p+1)/2} M_{np}\leq C_1 D_1^n M_{q\lfloor np/q\rfloor}$$
 for some $C_1,D_1>0$, and by Lemma \ref{lema2prima}
 $$ M_{q\lfloor np/q\rfloor}\leq B^{\lfloor np/q\rfloor q(q+1)/2}M^q_{\lfloor np/q\rfloor}\leq C_2 D_2^n M^q_{\lfloor np/q\rfloor}$$
 for some $C_2,D_2>0$.
 
 Hence
 $$
 M_n^{p/q}\leq C D^n M_{\lfloor np/q \rfloor}\quad \textrm{for every}\quad n\in\N_0
 $$
 with $C=(C_1 C_2)^{1/q}$ and $D=(D_1D_2)^{1/q}$.
\end{proof}


\begin{thebibliography}{99}
\bibitem{balser} W. Balser, Formal power series and linear systems of meromorphic ordinary differential equations, Universitext, Springer-Verlag, New York, 2000.
\bibitem{BY} W. Balser, M. Yoshino, \textit{Gevrey order of formal power series solutions of inhomogeneous partial differential equations with constant
coefficients}, Funkcial. Ekvac. 53 (2010), 411--434.
\bibitem{ichinobeadachi} K. Ichinobe, S. Adachi, \textit{On $k$-summability of formal solutions to the Cauchy problems for some linear $q$-difference-differential equations}, Complex Differential and Difference Equations, De Gruyter Proceedings in Mathematics (2020), 447--464.
\bibitem{immink} G.K. Immink, \emph{Exact asymptotics of nonlinear difference equations with levels 1 and 1+,} Ann. Fac. Sci. Toulouse T.XVII (2) (2008), 309--356.
\bibitem{lastramaleksanz} A. Lastra, S. Malek, J. Sanz, \emph{Summability in general Carleman ultraholomorphic classes}, J. Math. Anal. Appl. 430 (2015), 1175--1206. 
\bibitem{M} S. Michalik, \textit{Analytic solutions of moment partial differential equations with constant coefficients}, Funkcial. Ekvac. 56 (2013), 19--50.
\bibitem{michalik13jmaa} S. Michalik, \textit{Summability of formal solutions of linear partial differential equations with divergent initial data}, J. Math. Anal. Appl. 406 (2013), 243--260.
\bibitem{michalik17fe} S. Michalik, \textit{Analytic summable solutions of inhomogeneous moment partial differential equations}, Funkcial. Ekvac. 60 (2017), 325--351.
\bibitem{michaliksuwinska}  S. Michalik, M. Suwi\'nska, \textit{Gevrey estimates for certain moment partial differential equations}, Complex Differential and Difference Equations, De Gruyter Proceedings in Mathematics (2020), 391--408.
\bibitem{michaliktkacz} S. Michalik, B. Tkacz, \emph{The Stokes Phenomenon for some moment partial differential equations}, J. Dyn. Control Syst. 25 (2019),  573--598. 
\bibitem{javi17} J. Sanz, \emph{Asymptotic analysis and summability of formal power series}, Analytic, algebraic and geometric aspects of differential equations, Trends Math., Birkh\"auser/Springer, Cham, (2017) 199--262.
\bibitem{TY} H. Tahara, H. Yamazawa, \textit{Multisummability of formal solutions
to the Cauchy problem for some linear partial differential equations},
J. Differential Equations 255 (2013), 3592--3637.

\end{thebibliography}
\end{document}